\documentclass[11pt]{article}
\usepackage{amsmath,amssymb,amsthm}
\usepackage{mathrsfs}
\usepackage{geometry}
\usepackage{enumerate}
\usepackage{enumitem}
\newtheorem{thm}{Theorem}[section]
\newtheorem{lm}[thm]{Lemma}

\newtheorem{conj}{Conjecture}[section]
\newtheorem{pro}[thm]{Proposition}

\newtheorem{cor}[thm]{Corollary}

\usepackage{latexsym,amsmath,amssymb,amsfonts,epsfig,graphicx,cite,psfrag}
\usepackage{eepic,color,colordvi,amscd}
\usepackage{ebezier}
\usepackage{verbatim}

\usepackage{subfigure}

\newcommand{\F}{\mathcal{F}}
\newcommand{\E}{\mathcal{E}}
\newcommand{\A}{\mathcal{A}}

\newcommand{\B}{\mathcal{B}}

\newcommand{\pf}{\noindent {\bf Proof.} }

\begin{document}

\title{A stability result for almost perfect matchings\footnote{Supported by National Key R\&D Program of China under grant No. 2023YFA1010203 and  the National Natural
Science Foundation of China under grant No.12271425}}

\author{
Mingyang Guo, Hongliang Lu\footnote{Corresponding email: luhongliang215@sina.com} \\School of Mathematics and Statistics\\
Xi'an Jiaotong University\\
Xi'an, Shaanxi 710049, China
}

\date{}

\maketitle

\begin{abstract}
	Let $n,k,s$ be three integers and $\beta$ be a sufficiently small positive number such that $k\geq 3$, $0<1/n\ll \beta\ll 1/k$ and $ks+k\leq n\leq (1+\beta)ks+k-2$.  A $k$-graph is called non-trivial if it has no isolated vertex. In  this paper, we determine the maximum number of edges in a non-trivial $k$-graph with $n$ vertices and matching number at most $s$. This result confirms a conjecture proposed by Frankl (On non-trivial families without a perfect matching, \emph{European J. Combin.}, \textbf{84} (2020), 103044) for the case when $s$ is sufficiently large.
\end{abstract}

\section{Introduction}

For integers $n>k\geq 1$, let $[n]:=\{1,2,\cdots,n\}$ be the standard $n$ element set and $\binom{[n]}{k}=\{T\subseteq [n]:|T|=k\}$ be the collection of all its $k$-subsets. A \textit{$k$-uniform hypergraph} defined on $[n]$ is a family $\F$ such that $\F\subseteq \binom{[n]}{k}$. A $k$-uniform hypergraph is also called a \emph{$k$-graph}. For a $k$-graph $\F\subseteq \binom{[n]}{k}$, let $\nu(\F)$ denote the \textit{matching number} of $\F$, that is, the maximum number of pairwise disjoint members of $\F$.


Let $n,s,k$ be three positive integers such that $k\geq2$ and $n\geq ks+k-1$. For each $i\in [k]$, define
$$\A_i(n,k,s):=\left\{A\in\binom{[n]}{k}:|A\cap[(s+1)i-1]|\geq i\right\}.$$
In 1965, Erd\H os \cite{E65} asked for the determination of the maximum possible number of edges that can appear in any $k$-graph $\F$ with $\nu(\F)\leq s$. He conjectured that $\A_1(n,k,s)$ and $\A_k(n,k,s)$ are the two extremal constructions of this problem.
\begin{conj}[Erd\H os Matching Conjecture \cite{E65}]\label{conj}
	Let $n,s,k$ be three positive integers such that $k\geq2$ and $n\geq k(s+1)-1$. If $\F$ is a $k$-graph on $[n]$ and $\nu(\F)\leq s$, then
	\begin{equation*}
		|\F|\leq\max\left\{\binom{n}{k}-\binom{n-s}{k},\binom{k(s+1)-1}{k}\right\}.
	\end{equation*}
\end{conj}


There have been recent
activities on the Erd\H{o}s Matching Conjecture, see \cite{AFHRRS,BDE,EKR,EG,E65,F13,F,F17,FLM12,FK19,FK,HLS,LM,LYY,FRR,KK}.  The Erd\H{o}s Matching Conjecture was verified by Erd\H os and Gallai \cite{EG} for $k=2$. For $k\geq 3$, it was proved by Bollob\'as, Daykin and Erd\H os \cite{BDE} for $n>2k^3s$.  Subsequently, Huang, Loh and Sudakov \cite{HLS} settled the conjecture for $n>3k^2(s+1)$. In 2013, Frankl \cite{F13} verified the conjecture for $n\geq(2s+1)k-s$. Recently, Frankl and Kupavskii \cite{FK} proved the conjecture for $n\geq \frac{5}{3}sk-\frac{2}{3}s$ and sufficiently large $s$.
As for the special case of $k=3$, Frankl, R\"odl and Ruci\'nski \cite{FRR} proved the conjecture for $n\geq 4s$.  In particular, the Erd\H{o}s Matching Conjecture  was settled  for $k = 3$ and sufficiently large $n$ in \cite{LM}, and finally, it was completely resolved  for $k=3$ in \cite{F17}.

A \emph{vertex cover} in a $k$-graph $\F$ is a set of vertices which intersects all edges of $\F$. We use $\tau(\F)$ to denote the minimum size of a vertex cover in $\F$. Bollob\'as, Daykin and Erd\H os \cite{BDE} proved a stability result of Conjecture \ref{conj} for $n>2k^3s$.

\begin{thm}[Bollob\'as, Daykin and Erd\H os \cite{BDE}]\label{bollobas}
	Let $n,s,k$ be three positive integers such that $k\geq2$ and $n> 2k^3s$. If $\F$ is a $k$-graph on $[n]$ and  $\nu(\F)\leq s<\tau(\F)$, then
	\begin{equation}\label{hmkedge}
		|\F|\leq\binom{n}{k}-\binom{n-s}{k}-\binom{n-s-k}{k-1}+1.
	\end{equation}
\end{thm}

\noindent\textbf{Remark:}
The condition (\ref{hmkedge}) is tight. Define $$\B(n,k,s):=\left\{B\in\binom{[n]}{k}:B\cap[s-1]\neq\emptyset\right\}
\cup\left\{S\right\}\cup\left\{B\in\binom{[n]}{k}:s\in B, B\cap S\neq\emptyset\right\},$$ where $S:=\{s+1,\ldots,s+k\}$. Note that $|\B(n,k,s)|
=\binom{n}{k}-\binom{n-s}{k}-\binom{n-s-k}{k-1}+1$, $\nu(\B(n,k,s))=s$ and $\tau(\B(n,k,s))=s+1$.

Theorem \ref{bollobas} shows that if $\nu(\F)\leq s$ and $|\F|>|\B(n,k,s)|$, then $\tau(\F)\leq s$. That is, $\F$ is a subgraph of $\A_1(n,k,s)$.

Frankl and  Kupavskii \cite{FK19} improved the result in Theorem \ref{bollobas} by proving that for $k\geq3$ and either $n\geq(s+\max\{25,2s+2\})k$ or $n\geq(2+o(1))sk$, where $o(1)$ is with respect to $s\rightarrow\infty$, if $\F$ is a $k$-graph with $\nu(\F)=s$ and $\tau(\F)>s$, then $|\F|\leq\binom{n}{k}-\binom{n-s}{k}-\binom{n-s-k}{k-1}+1$. For the case $s=1$, a classical result of Hilton and Milner \cite{HM} shows that for a $k$-graph $\F$, if $\nu(\F)=1$ and $\tau(\F)> 1$, then $|\F|\leq \binom{n-1}{k-1}-\binom{n-k-1}{k-1}+1$ for $n>2k$. This can be seen as a stability result of the famous Erd\H os-Ko-Rado theorem \cite{EKR}.

Frankl and Kupavskii \cite{FK19} conjectured there are several different extremal graphs for the stability problem and proposed the following conjecture.
\begin{conj}[Frankl and Kupavskii \cite{FK19}]\label{conj1}
Suppose that $\F$ is a $k$-graph on vertex set $[n]$. If $\nu(\F)=s$ and $\tau(\F)>s$, then
\begin{equation*}
   \F\leq \max\{|\A_2(n,k,s)|,\ldots,|\A_{k}(n,k,s)|,|\B(n,k,s)|\}.
\end{equation*}
\end{conj}

Guo, Lu and Mao \cite{GLM} confirm Conjecture \ref{conj1} for $k=3$ and sufficiently large $n$.

A majority of the works on Erd\H os Matching Conjecture deal with the regimes when $\A_1(n,k,s)$ is larger than $\A_k(n,k,s)$. Frankl \cite{F} proved that if a $k$-graph $\F$ on $[n]$ satisfies $\nu(\F)\leq s$, then $|\F|\leq |\A_k(n,k,s)|$ for $k(s+1)\leq n\leq (k+k^{-2k-1}/2)(s+1)$. This is the first result proving Erd\H os Matching Conjecture for a range where $\A_k(n,k,s)$ is larger than $\A_1(n,k,s)$. Recently, Kolupaev and Kupavskii \cite{KK} improved the bounds by proving that Erd\H os Matching Conjecture holds for $k(s+1)\leq n\leq (k+1/(100k))(s+1)$.

Define
\begin{equation*}
	\begin{split}
\E_0(n,k,s):=\binom{[ks+k-2]}{k}\cup\left\{A\in \binom{[n]}{k}: A\setminus [ks+k-2]=\{ks+k-1\}, A\cap [k-1]\neq \emptyset\right\}\\
\cup\left\{[k-1]\cup \{x\}:ks+k\leq x\leq n\right\}
\end{split}
\end{equation*}
and
\begin{equation*}
	\begin{split}
		\E_1(n,k,s):=\binom{[ks+k-2]}{k}
		\cup\left(\bigcup_{x\in [ks+k-1,n]}\left\{A\in \binom{[n]}{k}: A\setminus [ks+k-2]=\{x\}, 1\in A\right\}\right).
	\end{split}
\end{equation*}

We followed the notation of \cite{F2020}. A $k$-graph $\F$ called \textit{non-trivial} if there is no isolated vertex in $\F$. Frankl \cite{F2020} proved the following theorem.

\begin{thm}[Frankl \cite{F2020}]\label{clique-sta}
	For any integer $k\geq 2$, there exists a positive integer $s_0$ such that the following holds. Let $s$ be a positive integer such that $s>s_0$. Suppose that $\F$ is a non-trivial $k$-graph on $(ks+k)$ vertices. If $\nu(H)\leq s$, then
	\begin{equation*}
		|\F|\leq|\E_{0}(ks+k,k,s)|.
	\end{equation*}
	
\end{thm}
This result shows that if a $k$-graph $\F$ with $(ks+k)$ vertices satisfies $\nu(\F)\leq s$ and $|\F|>|\E_{0}(ks+k,k,s)|$, then $\F$ is a subgraph of $\A_k(ks+k,k,s)$. Thus the result can be seen as a stability result on Erd\H os Matching Conjecture for the case when $\A_k(n,k,s)$ is larger than $\A_1(n,k,s)$. For $n>ks+k$, Frankl \cite{F2020} proposed the following conjecture.

\begin{conj}[Frankl \cite{F2020}]\label{stability-clique-conj}
	For any integer $\ell\geq 1$, there exists a positive integer $k_0$ such that the following holds. Let $n,s,k$ be integers such that $s\geq 5$, $k>k_0$ and $n=k(s+1)+\ell$. Suppose that $\mathcal{F}\subseteq \binom{[n]}{k}$. If $\mathcal{F}$ is non-trivial and $\nu(\mathcal{F})\leq s$, then
	\begin{equation*}
		|\mathcal{F}|\leq |\E_{0}(n,k,s)|.
	\end{equation*}
\end{conj}

In this paper, we focus on the stability problem on Erd\H os Matching Conjecture for a range where $\A_k(n,k,s)$ is larger than $\A_1(n,k,s)$. By $x\ll y$ we mean that for any $y>0$ there exists $x_0>0$ such that for any $x<x_0$ the following statement holds. We omit the floor and ceiling functions when they do not affect the proof. The following theorem is our main result.

\begin{thm}\label{main-thm}
	For any integer $k\geq 3$, there exist $\beta>0$ and a positive integer $s_0$ such that the following holds. Let $n,s$ be two integers such that $s>s_0$ and $ks+k\leq n\leq (1+\beta)ks+k-2$. Suppose that $\mathcal{F}\subseteq \binom{[n]}{k}$. If $\mathcal{F}$ is non-trivial and $\nu(\mathcal{F})\leq s$, then
	\begin{equation*}
		|\mathcal{F}|\leq \max\{|\mathcal{E}_0(n,k,s)|,|\mathcal{E}_1(n,k,s)|\}.
	\end{equation*}
\end{thm}

Theorem \ref{main-thm} implies the following result, which confirms Conjecture \ref{stability-clique-conj} for sufficiently large $s$.
\begin{cor}
		For any integers $k\geq 3$ and $\ell\geq 1$, there exists a positive integer $s_0$ such that the following holds. Let $n,s$ be two integers such that $s>s_0$ and $n=k(s+1)+\ell$. Suppose that $\mathcal{F}\subseteq \binom{[n]}{k}$. If $\mathcal{F}$ is non-trivial and $\nu(\mathcal{F})\leq s$, then
		\begin{equation*}
			|\F|\leq \left\{
			\begin{aligned}
				&|\E_{0}(n,k,s)|, \ \text{if} \ k> \ell+3;\\
				&|\E_{1}(n,k,s)|, \ \text{if} \ k\leq \ell+3.
			\end{aligned}
			\right.
		\end{equation*}
\end{cor}
\pf
By Theorem \ref{main-thm}, there exists a positive integer $s_0'$ such that $|\mathcal{F}|\leq \max_{i\in\{0,1\}}\{|\mathcal{E}_i(n,k,s)|\}$ for $s>s_0'$. Note that
\begin{equation*}
	|\mathcal{E}_0(n,k,s)|=\binom{ks+k-2}{k}+\binom{ks+k-2}{k-1}-\binom{ks-1}{k-1}+\ell+1
\end{equation*}
and
\begin{equation*}
	|\mathcal{E}_1(n,k,s)|=\binom{ks+k-2}{k}+(\ell+2)\binom{ks+k-3}{k-2}.
\end{equation*}
For $k> \ell+3$, one can see that there exists an integer $s_0''$ such that for $s>s''_0$,
\begin{equation*}
	\begin{split}
		|\mathcal{E}_0(n,k,s)|-|\mathcal{E}_1(n,k,s)|&=\binom{ks+k-2}{k-1}-\binom{ks-1}{k-1}+\ell+1-\left(\ell+2\right)\binom{ks+k-3}{k-2}\\
		&> (k-1)\binom{ks-1}{k-2}-\left(\ell+2\right)\binom{ks+k-3}{k-2}\\
		&= \binom{ks-1}{k-2}-\left((\ell+2)\binom{ks+k-3}{k-2}-(k-2)\binom{ks-1}{k-2}\right)\\
		&\geq\binom{ks-1}{k-2}-(k-2)\left(\binom{ks+k-3}{k-2}-\binom{ks-1}{k-2}\right)\\
		&>\binom{ks-1}{k-2}-(k-2)^2\binom{ks+k-3}{k-3}\\
		&>0.
	\end{split}
\end{equation*}

For $k\leq \ell+3$, one can see that 
\begin{equation*}
	\begin{split}
		&|\mathcal{E}_1(n,k,s)|-|\mathcal{E}_0(n,k,s)|\\
		=&(\ell+2)\binom{ks+k-3}{k-2}-\left(\binom{ks+k-2}{k-1}-\binom{ks-1}{k-1}+\ell+1\right)\\
		=&(\ell+2)\binom{ks+k-3}{k-2}-\left(\sum_{i=1}^{k-1}\binom{ks+k-2-i}{k-2}+\ell+1\right) \\
		=&((\ell+2)-(k-1))\left(\binom{ks+k-3}{k-2}-1\right)+\sum_{i=1}^{k-1}\left(\binom{ks+k-3}{k-2}-\binom{ks+k-2-i}{k-2}\right)-(k-2)\\
		\geq& 0
	\end{split}
\end{equation*}
holds for all $k\geq 3$ and $s\geq 1$. Let $s_0=\max\{s_0',s_0''\}$, then
\begin{equation*}
	|\F|\leq \left\{
	\begin{aligned}
		&|\E_{0}(n,k,s)|, \ \text{if} \ k> \ell+3;\\
		&|\E_{1}(n,k,s)|, \ \text{if} \ k\leq \ell+3.
	\end{aligned}
	\right.
\end{equation*}
for $s>s_0$.
\qed


The following corollary can be seen as a stability result on Erd\H os Matching Conjecture for a range where $\A_k(n,k,s)$ is larger than $\A_1(n,k,s)$. Let $\mathcal{F}\subset \binom{[n]}{k}$ be a $k$-graph. For a subset $S\subseteq [n]$, we use $\F[S]$ to denote the sub-hypergraph with vertex set $S$ and edge set $\{F\in \F : F\subseteq S\}$.
\begin{cor}
	For any integer $k\geq 3$, there exist $\beta>0 $ and a positive integer $s_0$ such that the following holds. Let $n,s$ be two integers such that $s>s_0$ and $ks+k\leq n\leq (1+\beta)ks+k-2$. Suppose that $\mathcal{F}\subseteq \binom{[n]}{k}$. If $\nu(\mathcal{F})\leq s$ and
	\begin{equation*}
		|\mathcal{F}|> \max\{|\mathcal{E}_0(n,k,s)|,|\mathcal{E}_1(n,k,s)|\},
	\end{equation*}
	then $\F$ is a subgraph of $\A_k(n,k,s)$.
\end{cor}
\pf
Let $I$ be the set of isolated vertices in $\F$ and $S=[n]\setminus I$. Suppose that $\F$ is not a subgraph of $\A_k(n,k,s)$, then $|I|\leq n-(ks+k)$. Thus $\F[S]$ is a non-trivial $k$-graph on $S$ with $n\geq |S|\geq ks+k$. By Theorem \ref{main-thm}, there exists an integer $s_0$ such that
$$|\F|= |\F[S]|\leq \max_{i\in\{0,1\}}\{|\mathcal{E}_i(|S|,k,s)|\}\leq \max_{i\in\{0,1\}}\{|\mathcal{E}_i(n,k,s)|$$ for $s>s_0$, a contradiction.
\qed

\section{shifting}

Let $\F$ be a $k$-graph on vertex set $[n]$. For vertices $1\leq x<y\leq n$, we define the \emph{$(x,y)$-shift} $S_{x,y}$ by $S_{x,y}(\F)=\{S_{x,y}(F):F\in \F\}$, where
$$S_{x,y}(F)=\begin{cases}
	F\setminus \{y\}\cup \{x\}, & \text{if $y\in F$, $x\notin F$ and $F\setminus \{y\}\cup \{x\}\notin \F$;}\\
	F, &\text{otherwise.}
\end{cases}$$
The following well-known result can be found in \cite{F95}.
\begin{lm}\label{shift-prop}
	The $(x,y)$-shift satisfies the following properties.
	\begin{enumerate}[itemsep=0pt,parsep=0pt,label=$($\roman*$)$]
		\item $|\F|=|S_{x,y}(\F)|$ and $|F|=|S_{x,y}(F)|$ for any $F\in \F$,
		\item $\nu(S_{x,y}(\F))\leq \nu(\F)$.
	\end{enumerate}
\end{lm}
For a $k$-graph $\F$, define $N_{\F}(x)=\{F\setminus\{x\}:x\in F\in \F\}$, $N_{\F}(x,y)=\{F\setminus\{x,y\}:x,y\in F\in \F\}$ and $N_{\F}(x,\overline{y})=\{F\setminus\{x\}:x\in F\in \F, y\notin F\}$.



\begin{pro}[Frankl \cite{F2020}]\label{shift-stop}
	Suppose that $\F\subseteq \binom{[n]}{k}$, then the following hold.
	\begin{enumerate}[itemsep=0pt,parsep=0pt,label=$($\roman*$)$]
		\item If $|N_{\F}(x)|\geq |N_{\F}(y)|$ and $\widetilde{\mathcal{F}}:=S_{x,y}(\F)\neq \F$, then
		\begin{equation}\label{meaningful-shifts}
			\sum_{i\in [n]}|N_{\widetilde{\mathcal{F}}}(i)|^2>\sum_{i\in [n]}|N_{\F}(i)|^2.
		\end{equation}
		\item If $\F$ is non-trivial and $\widetilde{\mathcal{F}}=S_{x,y}(\F)$ has isolated vertex, then $|N_{\F}(x,y)|=\emptyset$ and $|N_{\F}(x,\overline{y})|+|N_{\F}(y,\overline{x})|\leq \binom{n-2}{k-1}$.
	\end{enumerate}
\end{pro}

Let us say that the $(x,y)$-shift $S_{x,y}$ is \emph{meaningful} if (\ref{meaningful-shifts}) holds. Noting the obvious upper bound
$$\sum_{x\in[n]}|N_{\F}(x)|^2\leq n|\F|^2,$$
one concludes that there can be only a \emph{limited} number of successive meaningful shifts. For a subset $Y\subset[n]$, we say that $\F\subseteq\binom{[n]}{k}$ is \emph{shifted} on $Y$ if $S_{i,j}(\F)=\F$ holds for all $1\leq i<j\leq n$, $\{i,j\}\subset Y$. It is not difficult to see that if $\F$ is shifted on $Y$, then for any subsets $\{u_1,\ldots,u_k\}, \{v_1,\ldots,v_k\}\subset Y$ such that $u_i\leq v_i$ for $i\in[k]$, $\{v_1,\ldots,v_k\}\in \F$ implies $\{u_1,\ldots,u_k\}\in \F$. Let $\omega(\F)$ be the number of vertices in the largest complete subgraph of $\F$.
\begin{lm}\label{large-clique}
	For any integer $k\geq 3$ and positive real $\varepsilon< 1/k$, there exists $s_0(k,\varepsilon)$ such that the following holds. Let $0<\eta<\frac{\varepsilon^k}{3^{k+2}k!}$ and $n,s$ be positive integers such that $s>s_0$ and $(1-\eta)ks\leq n\leq (1+\eta)ks+k-2$. Suppose that $\F\subseteq\binom{[n]}{k}$ and $|\F|>\binom{ks+k-2}{k}-3\eta (ks)^k$. If $\F$ is shifted on $[n]$, then $\omega(\F)\geq (1-\varepsilon)ks $.
\end{lm}
\pf
Suppose that $\omega(\F)<(1-\varepsilon)ks$. Let $S=[n]\setminus [(1-\varepsilon)ks]$, then $\F[S]=\emptyset$ since $\F$ is shifted on $[n]$. Thus $|\F|<\binom{n}{k}-\binom{|S|}{k}$. Note that $|S|\geq n-(1-\varepsilon)ks> (\varepsilon-\eta) ks\geq \varepsilon ks/2$. Combining with $|\F|>\binom{ks+k-2}{k}-3\eta (ks)^k$, we have $\binom{n}{k}-\binom{|S|}{k}>\binom{ks+k-2}{k}-3\eta (ks)^k$. That is
\begin{equation}\label{inequality1}
	\binom{n}{k}-\binom{ks+k-2}{k}+3\eta (ks)^k>\binom{|S|}{k}.
\end{equation}
Since
\begin{equation*}
	\begin{split}
		\binom{n}{k}-\binom{ks+k-2}{k}\leq &\binom{(1+\eta)ks+k-2}{k}-\binom{ks+k-2}{k}\\
		\leq &\eta ks\binom{(1+\eta)ks+k-2}{k-1}\\
		\leq &\eta ks\binom{(1+2\eta)ks}{k-1}\\
		\leq&\frac{(1+2\eta)^{k-1}\eta(ks)^{k}}{(k-1)!}\\
		\leq & \frac{2^{k-1}\eta(ks)^{k}}{(k-1)!}
	\end{split}
\end{equation*}
and
$$\binom{|S|}{k}\geq \frac{(\varepsilon ks/2-k)^k}{k!}\geq \frac{\varepsilon^k(ks)^k}{3^kk!},$$
inequality (\ref{inequality1}) implies
$$\frac{\varepsilon^k}{3^kk!}\leq \frac{2^{k-1}\eta}{(k-1)!}+3\eta\leq 5\eta,$$
a contradiction.
\qed

Using some ideas from \cite{F2020}, we prove that if the number of edges in $\F$ is large enough, then we can perform shifts on most of vertices without destroying non-triviality.
\begin{lm}\label{non-trivial-shift}
	For any integer $k\geq 3$ and positive real $\beta\leq \frac{1}{9k}$, there exists $s_0(k,\beta)$ such that the following holds. Let $n,s$ be integers such that $s>s_0$ and $ks+k\leq n\leq (1+\beta)ks+k-2$. Suppose that $\mathcal{F}\subseteq \binom{[n]}{k}$, $|\mathcal{F}|\geq\binom{ks+k-2}{k}$ and $\mathcal{F}$ is non-trivial. Then there is a $k$-graph $\mathcal{F}^*\subseteq\binom{[n]}{k}$ obtained from $\mathcal{F}$ via successive shifts such that $\mathcal{F}^*$ is non-trivial and shifted on $[n-3\beta ks]$.
\end{lm}

\pf
We try and find $1\leq x<y\leq n-3\beta ks$ such that the $(x,y)$-shift $S_{x,y}$ is meaningful and we rename the elements of $[n]$ after  $(x,y)$-shift $S_{x,y}$ so that the new family $\widetilde{\mathcal{F}}$ satisfies $|N_{\widetilde{\mathcal{F}}}(1)|\geq \cdots\geq |N_{\widetilde{\mathcal{F}}}(n)|.$ After that, by abuse of notation we denote $\widetilde{\mathcal{F}}$ by $\mathcal{F}$ and continue to find another meaningful $(x',y')$-shift $S_{x',y'}$. After $(x',y')$-shift we denote the new family $\widetilde{\mathcal{F}}$ by $\F$ and rename the elements such that $|N_{\mathcal{F}}(1)|\geq \cdots\geq |N_{\mathcal{F}}(n)|$.
In view of Proposition \ref{shift-stop} this process is going to terminate after a finite number of steps. The only thing that we have to show is that $\widetilde{\mathcal{F}}=S_{x,y}(\mathcal{F})$ is non-trivial for $1\leq x<y\leq n-3\beta ks$.

If $N_{\widetilde{\mathcal{F}}}(y)=\emptyset$, then $|N_{\F}(x,y)|=\emptyset$ and $|N_{\F}(x,\overline{y})|+|N_{\F}(y,\overline{x})|<\binom{n-2}{k-1}$ by Proposition \ref{shift-stop}. It follows that
$$|N_{\mathcal{F}}(x)|+|N_{\mathcal{F}}(y)|=2|N_{\F}(x,y)|+|N_{\F}(x,\overline{y})|+|N_{\F}(y,\overline{x})|<\binom{n-2}{k-1}.$$
Thus
\begin{equation}\label{half-degree1}
N_{\mathcal{F}}(v)\leq N_{\mathcal{F}}(y)<\frac{1}{2}\binom{n-2}{k-1}
\end{equation}
for each $v\in \{y+1,\ldots,n\}$.

Suppose that $y\leq n-3\beta ks$ and $N_{\widetilde{\mathcal{F}}}(y)=\emptyset$, we are going to get a contradiction. To this end we introduce the notation
\begin{equation*}
	\alpha_j=\binom{n-1}{k-1}-|N_{\F}(j)|.
\end{equation*}
Note that $\alpha_x$ counts the number of missing sets in $N_{\mathcal{F}}(x)$, namely the sets $G\in \binom{[n]}{k}$ with $x\in G\notin \mathcal{F}$. Let $W=\{n-3\beta ks+1,\ldots,n\}$. By inequality (\ref{half-degree1}), for each $v\in W$
\begin{equation}\label{half-degree2}
	\alpha_v>\binom{n-1}{k-1}-\frac{1}{2}\binom{n-2}{k-1}>\frac{1}{2}\binom{n-1}{k-1}.
\end{equation}

Let $\alpha$ denote the total number of missing sets for $\mathcal{F}$. On the one hand,
\begin{equation}
	\alpha=\binom{n}{k}-|\F|< \beta ks\binom{n-1}{k-1}
\end{equation}
since $|\mathcal{F}|\geq \binom{ks+k-2}{k}$. On the other hand,
by inclusion-exclusion principle and inequality (\ref{half-degree2}), we have
\begin{equation}
	\begin{split}
		\alpha>&\sum_{i\in W}\alpha_i-\binom{|W|}{2}\binom{n-2}{k-2}\\
		\geq&\frac{|W|}{2}\binom{n-1}{k-1}-\binom{|W|}{2}\binom{n-2}{k-2}\\
		=&\left(\frac{|W|}{2}\frac{n-1}{k-1}-\binom{|W|}{2}\right)	\binom{n-2}{k-2}\\
		>&	\frac{|W|}{2}\left(\frac{n-1}{k-1}-|W|\right)\binom{n-2}{k-2}\\
		=&\frac{3\beta ks}{2}\left(\frac{n-1}{k-1}-3\beta ks\right)\binom{n-2}{k-2}\\
		\geq&\frac{3\beta ks}{2}\cdot\frac{2(n-1)}{3(k-1)}\cdot\binom{n-2}{k-2}\\
		=&\beta ks\binom{n-1}{k-1},
	\end{split}
\end{equation}
a contradiction.

\qed

\section{Proof of Theorem \ref{main-thm}}


In order to prove Theorem \ref{main-thm}, we need the following results. Two families $\mathcal{A}, \mathcal{B}$ are called \emph{cross-intersecting} if $A\cap B\neq\emptyset$ for all $A\in \mathcal{A}$, $B\in\mathcal{B}$. Let $\A_1,\A_2,\ldots,\A_t\subseteq \binom{[n]}{k}$. $\A_1,\A_2,\ldots,\A_t$ are called \emph{cross-intersecting} if $\mathcal{A}_i$ and  $\mathcal{A}_j$ are cross-intersecting for any $\A_i,\A_j\in \{\A_1,\A_2,\ldots,\A_t\}$ with $i\neq j$.



\begin{thm}[Shi, Frankl and Qian \cite{SFQ}]\label{t-crossintersecting}
	Let $n,k,t$ be positive integers with $n\geq 2k$ and $t\geq 2$. If $\A_1,\A_2,\ldots,\A_t\subseteq \binom{[n]}{k}$ are non-empty cross-intersecting families, then
	$$\sum_{i=1}^t|\A_i|\leq \max\left\{\binom{n}{k}-\binom{n-k}{k}+t-1,t\binom{n-1}{k-1}\right\}$$
	and the equality holds if and only if, up to isomorphism, one of the following holds:
	\begin{enumerate}[itemsep=0pt,parsep=0pt,label=$($\roman*$)$]
		\item $n>2k$, $\A_i=\{A\in \binom{[n]}{k}:1\in A\}$ for every $i\in \{1,2,\ldots,t\}$ if $\binom{n}{k}-\binom{n-k}{k}+t-1\leq t\binom{n-1}{k-1}$, or $\A_1=\{A\in \binom{[n]}{k}:[k]\cap A\neq \emptyset\}$ and $\A_j=\{[k]\}$ for every $j\in\{2,\ldots,t\}$ (up to rearrangement of families) if $\binom{n}{k}-\binom{n-k}{k}+t-1\geq t\binom{n-1}{k-1}$;
		\item $n=2k$, $t=2$, $\A_1\subseteq \binom{[n]}{k}$ with $0<|\A_1|<\binom{n}{k}$ and $\A_2=\binom{[n]}{k}\setminus \overline{\A_1}$;
		\item $n=2k$, $t\geq 3$, $\A_1\subseteq \binom{[n]}{k}$ is intersecting with $|\A_1|=\binom{n-1}{k-1}$ and $\A_2=\cdots=\A_t=\A_1$.
	\end{enumerate}
\end{thm}


Similar to Lemma 2 in \cite{LM}, Guo, Lu and Peng \cite{GLP} prove the following result. We say that $\F$ is \emph{s-saturated}, if $\nu(\F)\leq s$, but $\nu(\{F\}\cup \F)=s+1$ for every $F\notin \F$.
\begin{thm}\label{big-clique}
	For any integer $k\geq 3$, there exist $\varepsilon>0$ and a positive integer $n_0$ such that the following holds. Let $\F$ be a s-saturated $k$-graph on $n>n_0$ vertices and $s$ be an integer with $ n(1/k-1/(2k^2))-1< s\leq (n-k+1)/k$. If $\nu(\F)\leq s$, and $(1-\varepsilon)ks\leq\omega(\F)\leq ks+k-3$, then
	\begin{equation*}
		|\F|\leq\binom{ks+k-1}{k}-3.98\binom{(1-\varepsilon)ks}{k-1}+(2+8\varepsilon k^4)\binom{n}{k-1}.
	\end{equation*}
\end{thm}

\begin{proof}[Proof of Theorem \ref{main-thm}]

Suppose that $|\F|>\max_{i\in\{0,1\}}\{|\mathcal{E}_i(n,k,s)|\}>\binom{ks+k-2}{k}$, we will get a contradiction. Let $\varepsilon$ be a sufficiently small constant as in Theorem \ref{big-clique} and $\beta$ be a constant such that $3\beta<\frac{\varepsilon^k}{3^{k+2}k!}$. By Lemma \ref{non-trivial-shift}, there is a $k$-graph $\F'$ such that $\F'$ is non-trivial and shifted on $[n-3\beta ks]$. According to Lemma \ref{shift-prop}, $|\F'|=|\F|$ and $\nu(\F')\leq \nu(\F)$. So we may assume that $\F$ is non-trivial and shifted on $[n-3\beta ks]$ in our proof. Since $|\F|>\binom{ks+k-2}{k}$, we have
\begin{equation}
	\begin{split}
	\left|\F[n-3\beta ks]\right|>&\binom{ks+k-2}{k}-3\beta ks\binom{n-1}{k-1}\\
	>&\binom{ks+k-2}{k}-3\beta\cdot\frac{(1+2\beta)^{k-1}(ks)^k}{(k-1)!}\\
	>&\binom{ks+k-2}{k}-6\beta (ks)^k.
	\end{split}
\end{equation}
Note that $(1-3\beta)ks+k\leq n-3\beta ks\leq (1-2\beta) ks+k-2$. By Lemma \ref{large-clique}, $\omega(\F)>(1-\varepsilon)ks$. Since we can turn a $k$-graph without a matching of size $s+1$ to an $s$-saturated graph by adding edges, we may assume that $\F$ is $s$-saturated in the following proof.

We first claim that $\omega(\F)\leq ks+k-2$. Otherwise, suppose that there is a clique $U$ of size $ks+k-1$, then there is an edge $F$ not contained in $U$ since $\F$ is non-trivial. But there is a matching of size $s+1$ in $\F[U\cup F]$, a contradiction. For the case $\omega(\F)\leq ks+k-3$, by Theorem \ref{big-clique},
\begin{equation}
	\begin{split}
		|\F|\leq& \binom{ks+k-1}{k}-3.98\binom{(1-\varepsilon)ks}{k-1}+\left(2+8\varepsilon k^4\right)\binom{n}{k-1}\\
		<& \binom{ks+k-1}{k}-1.5\binom{(1-\varepsilon)ks}{k-1}\\
		<& \binom{ks+k-1}{k}-\binom{ks+k-2}{k-1}\\
		=&\binom{ks+k-2}{k}
	\end{split}
\end{equation}
for sufficiently large $s$ and small $\varepsilon$, a contradiction.

Now let us consider the case $\omega(\F)=ks+k-2$.
Let $U$ be the largest clique in $\F$ such that $|U|=ks+k-2$. Next we focus on vertices in $[n]\setminus U$. We claim that for any vertex $x\in [n]\setminus U$, $N_{\F}(x)\subset \binom{U}{k-1}$. Indeed, if there is an edge $F\in \F$ such that $x,y\in F$ and $x,y\in [n]\setminus U$, then $\F[U\setminus F]$ contains a matching $M$ of size $s$ since $|U\setminus F|=ks$. Thus $M\cup \{F\}$ is a matching of size $s+1$ in $\F$, which contradicts that $\nu(\F)\leq s$.

Let $U=[ks+k-2]$. So $N_{\mathcal{F}}(ks+k-1),\ldots,N_{\mathcal{F}}(n)\subseteq \binom{U}{k-1}$. Since $\nu(H)\leq s$, one can see that $N_{\mathcal{F}}(i),N_{\mathcal{F}}(j)$ are cross-intersecting for $ks+k-1\leq i<j\leq n$. Indeed, if there exist $i,j$ such that $ks+k-1\leq i<j\leq n$ and $N_{\mathcal{F}}(i),N_{\mathcal{F}}(j)$ are not cross-intersecting, then there exist $F_i\in N_{\mathcal{F}}(i)$ and $F_j\in N_{\mathcal{F}}(j)$ such that $F_i\cap F_j=\emptyset$. Thus $M_1:=\{F_i\cup\{i\},F_j\cup\{j\}\}$ is a matching of size two in $\F$. Since $|U\setminus V(M_1)|=k(s-1)$, $\F[U\setminus V(M_1)]$ contains a matching $M_2$ of size $s-1$. Thus $M_1\cup M_2$ is a matching of size $s+1$ in $\F$, which contradicts that $\nu(H)\leq s$.

Recall that $\F$ is non-trivial, then $N_{\mathcal{F}}(i)$ is non-empty for each $i\in \{ks+k-1,\ldots,n\}$. By Theorem \ref{t-crossintersecting},
\begin{equation*}
	\begin{split}
	\sum_{i=ks+k-1}^n|N_{\F}(i)|\leq\max\left\{\binom{ks+k-2}{k-1}-\binom{ks-1}{k-1}+n-ks-k+1,(n-ks-k+2)\binom{ks+k-3}{k-2}\right\}.
	\end{split}
\end{equation*}
Thus
\begin{equation*}
	\begin{split}
		|\F|&=\binom{ks+k-2}{k}+\sum_{i=ks+k-1}^n|N_{\F}(i)|\\
		&\leq \binom{ks+k-2}{k}+\max\left\{\binom{ks+k-2}{k-1}-\binom{ks-1}{k-1}+n-ks-k+1,(n-ks-k+2)\binom{ks+k-3}{k-2}\right\}\\
		&= \max_{i\in\{0,1\}}\{|\mathcal{E}_i(n,k,s)|\},
	\end{split}
\end{equation*}
a contradiction.

\end{proof}

\end{document}